\documentclass[smallextended,envcountsect]{svjour3}
\usepackage{amssymb, amsmath}
\usepackage{anysize, graphicx, float}

\usepackage[english]{babel}
\usepackage{hyperref}

\newcommand{\K}{\mathcal{K}}
\newcommand{\F}{\mathcal{F}}
\newcommand{\D}{\mathcal{D}}
\newcommand{\C}{\mathcal{C}}

\newcommand\PG{\mathrm{PG}}

\newcommand\GF{\mathrm{GF}}
\newcommand\Aut{\mathrm{Aut}}

\newcommand\Tr{\mathrm{Tr}}

\newcommand\N{\mathbb{N}}

\author {Frank De Clerck \and Stefaan De Winter\thanks{The second author is Postdoctoral Fellow of the Science Foundation - Flanders} \and Thomas Maes}

\title{Singer $8$-arcs  of Mathon type in $\PG(2,2^7)$}
 \institute{ %
Department of Mathematics,
Ghent University,
Krijgslaan 281--S22, B-9000 Ghent, Belgium.
\email{fdc@cage.ugent.be; sgdwinte@cage.ugent.be; tmmaes@cage.ugent.be}
}
\begin{document}
\maketitle

\begin{abstract}
In \cite{geometricapproach} De Clerck, De Winter and Maes counted the number of non-isomorphic Mathon maximal arcs of degree $8$ in $\PG(2,2^h)$, $h \neq 7$ and prime.  In this article we will show that in $\PG(2,2^7)$ a special class of Mathon maximal arcs of degree $8$ arises which admits a Singer group (i.e. a sharply transitive group) on the $7$ conics of these arcs.  We will give a detailed description of these arcs, and then count the total number of non-isomorphic Mathon maximal arcs of degree 8. Finally we show that the special arcs found in $\PG(2,2^7)$ extend to two infinite families of Mathon arcs of degree $8$ in $\PG(2,2^k)$, $k$ odd and divisible by $7$, while maintaining the nice property of admitting a Singer group. 
\end{abstract}

\section{Introduction}
\label{intro}
A $\{k;d\}$-arc $\K$, in a finite projective plane of order $q$, is a non-empty proper subset of $k$ points such that some line of the plane meets $\K$ in $d$ points, but no line meets $\K$ in more than $d$ points.  For given $q$ and $d$, $k$ can never exceed $q(d-1)+d$.  If equality holds $\K$ is called a \textit{maximal arc} of degree $d$, a degree $d$ maximal arc or simply, a maximal $d$-arc.  Equivalently, a maximal arc can be defined as a non-empty, proper subset of points such that every line meets the set in $0$ or $d$ points, for some $d$.  The set of points of an affine subplane of order $d$ of a projective plane of order $d$ is a trivial example of a $\{d^2;d\}$-arc, as well as a single point, being a $\{1;1\}$-arc of the  projective plane.  We will neglect  these trivial examples for the rest of this paper. \par
If $\K$ is a maximal $\{q(d-1)+d;d\}$-arc in a projective plane of order $q$, the set of lines external to $\K$ is a maximal $\{q(q-d+1)/d;q/d\}$-arc in the dual plane called the \textit{dual} of $\K$.  It follows that a necessary condition for the existence of a maximal $\{q(d-1)+d;d\}$-arc in a projective plane of order $q$ is that $d$ divides $q$.  Denniston showed that this necessary condition is sufficient in the Desarguesian projective plane $\PG(2,q)$ of order $q$ when $q$ is even \cite{MR0239991}.  Ball, Blokhuis and Mazzocca showed that no non-trivial maximal arcs exist in a Desarguesian projective plane of odd order \cite{MR1466573}.  \par
In \cite{MR1883870}, Mathon gave a construction method for maximal arcs in Desarguesian projective planes that generalized the previously known construction of Denniston \cite{MR0239991}.  We will begin by describing this construction method of Mathon.  \\

From now on let $q=2^h$ and let $\Tr$ denote the usual absolute trace map from the finite field $\GF(q)$ onto $\GF(2)$.  We represent the points of the Desarguesian projective plane $\PG(2,q)$ as triples $(a,b,c)$ over $\GF(q)$, and the lines as triples $[u,v,w]$ over $\GF(q)$. A point $(a,b,c)$ is incident with a line $\left[u,v,w\right]$ if and only if $au+bv+cw=0$.  For $\alpha, \beta \in \GF(q)$ such that $\Tr(\alpha \beta)=1$, and $\lambda \in \GF(q)^*= \GF(q) \setminus \{0\}$ we define $F_{\alpha,\beta,\lambda}$ to be the conic
\[
F_{\alpha,\beta,\lambda}=\{(x,y,z):\alpha x^2+xy+\beta y^2+\lambda z^2=0\}.
\]  
Let $\F$ be the set of all such conics.  Remark that all the conics in $\F$ have the point $F_0=(0,0,1)$ as their nucleus and that, due to the trace condition, the line $z=0$ is external to all conics.  \par
For given $\lambda \neq \lambda'$, define a composition
\[
F_{\alpha, \beta, \lambda} \oplus F_{\alpha', \beta', \lambda'}=F_{\alpha \oplus \alpha', \beta \oplus \beta', \lambda \oplus \lambda'}
\]
where the operator $\oplus$ is defined  as follows:
\[
\alpha \oplus \alpha'= \frac{\alpha \lambda + \alpha'\lambda'}{\lambda + \lambda'}, \>\> \beta \oplus \beta'= \frac{\beta\lambda + \beta'\lambda'}{\lambda + \lambda'}, \>\> \lambda \oplus \lambda'= \lambda + \lambda'.
\]
The following lemma was proved by Mathon in \cite{MR1883870}.

\begin{lemma}  \label{discon}
Two non-degenerate conics $F_{\alpha,\beta,\lambda}$, $F_{\alpha',\beta',\lambda'}$, $\lambda \neq \lambda'$ and their composition $F_{\alpha,\beta,\lambda} \oplus F_{\alpha',\beta',\lambda'}$ are mutually disjoint if $\Tr((\alpha \oplus \alpha')(\beta \oplus \beta'))=1$.
\end{lemma}

 Given some subset $\C$ of $\F$, we say $\C$ is \textit{closed} if for every $F_{\alpha,\beta,\lambda} \neq F_{\alpha',\beta',\lambda'} \in \C$, $F_{\alpha \oplus \alpha',\beta \oplus \beta',\lambda \oplus\lambda'} \in \C$.  We can now state Mathon's theorem.
\begin{theopargself}
\begin{theorem} [\cite{MR1883870}]  \label{mathon}
Let $\C\subset\F$ be a closed set of conics  in $\PG(2,q)$, $q$ even.  Then the union of the points on the conics of $\C$ together with their common nucleus $F_0$ is a degree $\vert \C \vert +1$ maximal arc in $\PG(2,q)$.
\end{theorem}
\end{theopargself}
We will sometimes call $F_0$ the nucleus of the maximal arc; by \cite{MR1997407} this is well defined. 
Note that a maximal arc of degree $d$ of Mathon type contains Mathon sub-arcs of degree $d'$ for all $d'$ dividing $d$ (see \cite{MR1883870}). Every maximal arc isomorphic to one as constructed above will be called a {\it maximal arc of Mathon type}. As we mentioned above, Mathon's construction is a generalization of a previously known construction of Denniston.  This can be seen as follows.  Choose $\alpha \in \GF(q)$ such that $\Tr(\alpha)=1$.  Let $A$ be a subset of $\GF(q)^{\star}$ such that $A \cup \{0\}$ is closed under addition.  Then the point set of the conics 
\[
\K_A=\{F_{\alpha,1,\lambda}:\lambda \in A\}
\] 
together with the nucleus $F_0=(0,0,1)$ is the set of points of a degree $\vert A \vert +1$ maximal arc in $\PG(2,q)$. Every maximal arc isomorphic to such an arc will be called a {\it maximal arc of Denniston type}.  The conics in $\K_A$ are a subset of the \textit{standard pencil of conics} given by
\[
\{F_{\alpha,1,\lambda}:\lambda \in \GF(q)^*\}.
\]  
This pencil partitions the points of the plane, not on the line $z=0$, and distinct from $(0,0,1)$, into $q-1$ disjoint  conics on the common nucleus $F_0=(0,0,1)$.  The line $z=0$ is often called the \textit{line at infinity} of the pencil and is denoted by $F_{\infty}$.  It has been proved \cite{MR1883870} that all degree 4 maximal Mathon arcs are necessarily of Denniston type.  \par

The following lemma was proved by Aguglia, Giuzzi and Korchmaros.  
\begin{theopargself}
\begin{lemma}[\cite{MR2443285}]  \label{aggiko}
Given any two disjoint conics $C_1$ and $C_2$ on a common nucleus.  Then there is a unique degree 4 maximal arc of Denniston type containing $C_1 \cup C_2$.
\end{lemma}
\end{theopargself}
In \cite{geometricapproach} this was generalized to a synthetic version of Mathon's construction. 

\begin{theorem}[Synthetic version of Mathon's theorem]  \label{synthmathon}
Given a degree $d$ maximal arc $M$ of Mathon type, consisting of $d-1$ conics on a common nucleus $n$, and a conic $C_d$ disjoint from $M$ with the same nucleus $n$, then there is a unique degree $2d$ maximal arc of Mathon type $\left<M,C_d\right>$ containing $M \cup C_d$.
\end{theorem}

In \cite{geometricapproach} this was used to count the number of non-isomorphic maximal arcs of Mathon type of degree $8$ in $\PG(2,2^p)$, with $p$ prime and $p\neq2,3,7$. The fact that our count did not work for $p=7$ suggested that something special might be going on in $\PG(2,2^7)$. In this paper we will see that this is indeed the case, as we will show that this specific plane admits two maximal  $8$-arcs of Mathon type with a particularly interesting automorphism group that do not exist in any of the other planes $\PG(2,2^p)$, $p$ prime.

Let us first have a look at the geometric structure of a maximal $8$-arc of Mathon type; this is based on \cite{MR1940336}. Note that if $\K$ is a maximal arc constructed from a closed set of conics $\C$ on a common nucleus, then the point set of that arc contains no non-degenerate conics apart from those of $\C$ (see \cite{MR1997407}). From Lemma \ref{aggiko} it immediately follows that every Mathon $8$-arc contains exactly $7$ Denniston $4$-arcs, and each two of these $7$ arcs have exactly one conic in common.  In fact, one easily sees that the structure with as point set the conics of $\K$, line set the degree $4$ subarcs of Denniston type, and the natural incidence is isomorphic to $\PG(2,2)$. In accordance with \cite{MR1940336} we define the lines at infinity of $\K$ to be the lines at infinity of each of the pencils determined by the degree $4$ subarcs. If $\K$ is of Denniston type there is a unique line at infinity, otherwise there are exactly $7$ distinct lines at infinity (see Theorem 2.2 of \cite{MR1940336} and the remark preceding it).  The next lemma shows that there always exists an involution stabilizing $\K$ and all of its conics.

\begin{theopargself}
\begin{lemma}[\cite{geometricapproach}]  \label{iota}
Let $\K$ be a maximal $8$ arc of Mathon type that is not of Denniston type. Then the $7$ lines at infinity of $\K$ are concurrent and there exists a unique involution stabilizing $\K$ and all conics contained in $\K$. This involution is an elation with center the point of intersection of the lines at infinity, and axis the line containing the nucleus of $\K$ and the center.
\end{lemma}
\end{theopargself}

In the specific case of $\PG(2,2^p)$, $p$ prime and $p \neq 2,3,7$ the following holds.

\begin{theopargself}
\begin{lemma}[\cite{geometricapproach}]  \label{C2}
Let $\K$ be a degree $8$ maximal arc of Mathon type that is not of Denniston type in $\PG(2,2^p)$, $p$ prime, $p\neq2,3,7$. Then $\mathrm{Aut}(\K)\cong C_2$.
\end{lemma}
\end{theopargself}

Using Theorem \ref{synthmathon}, the above lemma, and a counting argument that will be exploited also in this paper, it is possible to count the number of non-isomorphic degree $8$ arcs of Mathon type in $\PG(2,2^{2h+1}), 2h+1 \neq 7$ and prime.

\begin{theopargself}
\begin{theorem}[\cite{geometricapproach}]  \label{8count}
The number of non-isomorphic degree 8 maximal arcs of Mathon type in $\PG(2,2^{2h+1})$, $2h+1 \neq 7$ and prime, is exactly
\[
\frac{N}{14}(2^{2h-2}-1)((6h+3)N-1),
\]
where $N=\frac{2^{2h}-1}{3(2h+1)}$ is the number of non-isomorphic Denniston arcs of degree $4$ in $\PG(2,2^{2h+1})$.
\end{theorem}
\end{theopargself}

\section*{Remark} 

This theorem does not hold for $2h+1=7$ (note that in that case the obtained number is not even an integer), because Lemma \ref{C2} fails for $2h+1=7$.  We discuss below in more detail why this happens. Let $\K$ a degree $8$ maximal arc of Mathon type in $\PG(2,2^7)$, and let $\phi$ be a non-trivial automorphism of $\K$.  If $\phi$ stabilizes one of the degree 4 maximal subarcs of $\K$ it turns out (proof of Corollary \ref{C2} in \cite{geometricapproach}) that $\phi$ must be the unique involution $\iota$ described in Lemma \ref{iota}. Now suppose that $\phi$ does not stabilize any of the Denniston subarcs.  Since 7 is the only possible orbit length of $\phi$ on these subarcs it follows that the order of $\langle \phi \rangle$ has to be a multiple of 7.  Let the order of $\langle \phi \rangle$ be $k7$, with $k \in \N^{\star}$.  In that case $\vert \langle \phi \rangle_{\D} \vert=k$, where $\D$ is any of the $7$ Denniston subarcs of degree $4$.  Furthermore, since $\vert \Aut(\K)_{\D} \vert =2$ ( see \cite{geometricapproach}) we find that $k=2$.  This means that $\vert \Aut(\K) \vert=14$. Hence, in $\PG(2,2^7)$ a proper Mathon arc of degree $8$ could have a full automorphism group of order $2$ or of order $14$.  The latter type of arc is of specific interest, especially because of the subgroup of order $7$ cyclically permuting the conics of the arc.     


The previous remark suggests the existence of two classes of degree $8$ maximal arcs of Mathon type in $\PG(2,2^7)$.  The degree $8$ maximal arcs of Mathon that have a full automorphism group of order $2$   will be referred to as ``\textit{normal $8$-arcs}" in section \ref{8ArcCount}.  \\
Those admitting a full automorphism group of order $14$ will be called ``\textit{Singer $8$-arcs}".  \\  Also, with a little abuse of definition, Denniston arcs of degree $8$ that admit a group acting sharply transitively on their $7$ conics will be called ``\textit{Singer $8$-arcs}" as well in the next two sections.

We will now move to a detailed analysis of Mathon maximal arcs of degree $8$ in $\PG(2,2^7)$, and prove the existence of two classes of Singer $8$-arcs of Mathon type.

\section{Necessary conditions for the existence of a Singer arc}  \label{NecessCond}

Let $\D_1$ be a given  degree $4$ maximal arc of Denniston type in the standard pencil consisting of the conics $C_1,C_w,C_{w+1}$. (Note that every degree $4$ maximal arc of Denniston type is isomorphic to such an arc.) Due to Lemma \ref{aggiko} each conic $C$ disjoint from $\D_1$, and with nucleus $F_0$,  will give rise to a degree $8$ maximal arc of Mathon type (which might be of Denniston type).  \\
Let the additive subgroup $\{0,1,w,w+1\}, w \in \GF(2^7) \setminus \{0,1\}$, be the one associated to the maximal arc $\D_1$.  In other words we assume that the conics $C_1,C_w, C_{w+1}$ contained in $\D_1$ are given by the equation
\[
C_i:x^2+xy+y^2+iz^2=0,
\] 
where $i=1,w,w+1$.

If we now assume that $\K:=\left<\D_1,C\right>$ is a Singer $8$-arc, then necessarily all degree $4$ arcs of Denniston type contained in it have to be isomorphic, as these $7$ Denniston arcs will be cyclically permuted by the Singer group (= the cyclic group of order 7 permuting the conics, and hence the arcs). This explains why we will consider images of $\D_1$ that have exactly one conic in common with $\D_1$.

Consider the automorphism $\theta_{t,\sigma}$ of $\PG(2,2^7)$ given by 
\[
\theta_{t,\sigma}:x \mapsto A x^{\sigma}
\]
with
\begin{equation} \label{automw}
A:=\left(\begin{array}{ccc} 
		\sqrt{w}^{-\sigma}&0&0\\t&\sqrt{w}^{-\sigma}&0\\ \sqrt{\sqrt{w}^{-\sigma}t+t^2}&0&1
	\end{array}\right),
\end{equation} 
$\sigma \in \Aut(\GF(2^7))$ and $t \in \GF(2^7)$.  This automorphism will map $C_w$ onto $C_1$ while $(0,0,1)^{\theta}=(0,0,1)$ and $(0,1,0)^{\theta}=(0,1,0)$. This latter restriction on the automorphism $\theta_{t,\sigma}$ can be made without loss of generality since the $7$ lines at infinity of a Singer $8$-arc of Mathon type have to be concurrent; hence the above restriction simply chooses the line $X=0$ to be the axis of the unique elation stabilizing our $8$-arc under construction. Notice that we could equally well map $C_{w+1}$ onto $C_1$ or consider any other combination of two of the conics $C_1, C_w, C_{w+1}$ since our purpose is to find a Denniston arc of degree $4$ intersecting $\D_1$  in exactly one conic and being isomorphic to $\D_1$.  In \cite{geometricapproach} the authors obtained two trace conditions that are equivalent to this property. It is clear that these trace conditions still have to hold.  They can be written as:
\begin{eqnarray}  \label{Tr1}
\Tr\biggl[\frac{(1+w)t(\sqrt{w}^{-\sigma}+t)}{(w^{-\sigma}+w)}\biggr]=0
\end{eqnarray}
and
\begin{eqnarray}  \label{Tr2}
\Tr\biggl[\frac{wt(\sqrt{w}^{-\sigma}+t)}{(w^{-\sigma}+w+1)}\biggr]=0
\end{eqnarray}

In view of $\vert \Aut(\K)_{\D} \vert =2$ for any degree $4$ Denniston arc contained in $\K$ (see the above Remark  or \cite{geometricapproach}), we see that the only automorphisms mapping $\D_1$ onto $\D_2:=\D_1^{\theta_{t,\sigma}}$, while fixing $(0,1,0)$ and stabilizing $z=0$, are $\theta_{t,\sigma}$ and $\iota\theta_{t,\sigma}$ (where $\iota$ is the involution described in Lemma \ref{iota}). Hence, if we assume $\K$ to be a Singer $8$-arc of Mathon type, these $2$ automorphisms of $\PG(2,2^7)$ should belong to the automorphism group of $\K$. It is now natural to look at the action of powers of $\theta_{t,\sigma}$.

We now specialize to the action of this automorphism on the line $x=0$.

If we need the Singer group to act on the seven conics of the maximal $8$-arc, more specifically, if we want the seven conics to be cyclically permuted, the intersection points of each conic with the line $x=0$ should not only be distinct, but furthermore, if these intersection points are $(0,y_i,1)$, $i=1,\cdots,7$, then $\{0,y_1,\cdots,y_7\}$ should form an additive group of order $8$.  \\

We also remark that, in this case, the field automorphism $\sigma$ cannot be the identity.  This follows from the proof of Corollary 2 in \cite{geometricapproach} where the authors prove that the non-trivial automorphism $\phi$ mentioned above cannot belong to $\mathrm{PGL}(3,2^p)$.  \\

First of all we will calculate these intersection points.  The automorphism $\theta_{t,\sigma}$ acts on the points $(0,y,1), y \in \GF(2^7)$ contained on the line $x=0$ as follows:

\[
\left(\begin{array}{ccc} 
		\sqrt{w}^{-\sigma}&0&0\\t&\sqrt{w}^{-\sigma}&0\\ \sqrt{\sqrt{w}^{-\sigma}t+t^2}&0&1
	\end{array}\right)
\left(\begin{array}{c}
		0\\y\\1
	\end{array}\right)^{\sigma}=
\left(\begin{array}{c}
		0\\ \sqrt{w}^{-\sigma} y^{\sigma} \\ 1
	\end{array}\right),
\]

with $\sigma \in \Aut(\GF(2^7))$ and $t \in \GF(2^7)$.  Notice that the point $(0,\sqrt{w},1)$ is indeed mapped onto the point $(0,1,1)$.  In order to find all the intersection points, i.e. the images of the point $(0,1,1)$ under $\theta_{t,\sigma}$, $\theta_{t,\sigma}^2$, $\theta_{t,\sigma}^3$, ..., $\theta_{t,\sigma}^7$, we only need the element on position $(2,2)$ of the matrices $A$, $A.A^{\sigma}$, $A.A^{\sigma}.A^{\sigma^2}$,..., $A.A^{\sigma}.A^{\sigma^2}.A^{\sigma^3}.A^{\sigma^4}.A^{\sigma^5}.A^{\sigma^6}$.  These are, respectively, 
\[
\sqrt{w}^{-\sigma}, \sqrt{w}^{-(\sigma^2+\sigma)}, \sqrt{w}^{-(\sigma^3+\sigma^2+\sigma)},...,\sqrt{w}^{-(\sigma^6+...+\sigma^2+\sigma)}.   
\]

It follows that the seven intersection points are
\[
(0,1,1),(0,\sqrt{w}^{-\sigma},1),(0, \sqrt{w}^{-(\sigma^2+\sigma)},1),...,(0,\sqrt{w}^{-(\sigma^6+...+\sigma^2+\sigma)},1).
\]
We want to show that the set of elements 
\[
\{0,1,\sqrt{w}^{-\sigma}, \sqrt{w}^{-(\sigma^2+\sigma)}, \sqrt{w}^{-(\sigma^3+\sigma^2+\sigma)},...,\sqrt{w}^{-(\sigma^6+...+\sigma^2+\sigma)}\}
\] 
forms an additive group.  Therefore we will start by proving the following lemma.

\begin{lemma}  \label{sigmatox}
The set $\{0,1,\sqrt{w}^{-\sigma}, \sqrt{w}^{-(\sigma^2+\sigma)}, \sqrt{w}^{-(\sigma^3+\sigma^2+\sigma)},...,\sqrt{w}^{-(\sigma^6+...+\sigma^2+\sigma)}\}$ of elements in $\GF(2^7)$, with $\sigma$ any non-trivial automorphism  of  $\GF(2^7)$, can be written as the set 
\[
\{0,1,x,x^3,x^7,x^{15},x^{31},x^{63}\},
\]
where $x$ is one of the elements $\sqrt{w}^{-(\sigma^i+...+\sigma)}$.
\end{lemma}  
\begin{proof}
Let $\sigma = 2^k$, $k=1,...,6$.  Notice that there is exactly one integer $l\not\equiv 0$ (mod $7$) such that 
$\sigma^l=2^{kl}=2$, and that different $k$ yield different $l$. Set $x:=\sqrt{w}^{-(\sigma+...+\sigma^l)}$.

Now consider 

\[x^3=\sqrt{w}^{-3(\sigma+\sigma^2+...+\sigma^l)}.\]

We need to show that 
\begin{equation} \label{pairs(i,j)}
3(\sigma+\sigma^2+...+\sigma^l)=\sigma+\sigma^2+...+\sigma^{j_3}
\end{equation}
for some $j_3$.

If $l=1,2,3$ we see that $j_3=2, 4$ and $6$ respectively, satisfies equation (\ref{pairs(i,j)}).  If $l=4$ we find
\[
\sigma+\sigma^2+\sigma^3+\sigma^4+\sigma^5+\sigma^6+\sigma^7+\sigma^8
\]
on the left hand side of (\ref{pairs(i,j)}).

Now, using $\sigma^7=1$, and  the equality $\sigma^6+\sigma^5+...+\sigma+1=0$,  we see that, if $l=4$, $j_3=1$ is a solution to (\ref{pairs(i,j)}).  In an analogous way we can compute the values of $j_3$ that satisfy equation (\ref{pairs(i,j)}) for $l=5, 6$. In fact, using the same argument,  we see that $j_3 \equiv 2l \mod 7$ is the unique solution $\pmod 7$ to (\ref{pairs(i,j)}). In exactly the same way, we can also show that for each $l$ there are unique solutions (modulo $7$) to the following equations 

\[k(\sigma+\sigma^2+...+\sigma^l)=\sigma+\sigma^2+...+\sigma^{j_k},\]
with $k=7,15,31,63$.

 We get the following table. 

\begin{eqnarray} \label{(l,i,j)}
\begin{array}{|l|l|l|l|l|l|}
\hline
l  &  j_3 & j_7 & j_{15} & j_{31} & j_{63} \\ \hline
1 & 2    &   3     &    4        &      5     &     6     \\ \hline
2 & 4    &    6    &     1       &       3      &     5        \\ \hline
3 & 6    &     2   &      5      &     1       &      4         \\ \hline
4 &  1   &    5    &     2       &     6       &      3          \\ \hline
5 & 3    &    1    &      6      &      4      &      2          \\ \hline
6 & 5    &    4     &      3      &     2      &      1          \\ \hline
\end{array}
\end{eqnarray}

Clearly each row contains every non-trivial power of $\sigma$ exactly ones. This implies that for every $l$, with the given choice of $x$, our set is indeed representable as 
\[
\{0,1,x,x^3,x^7,x^{15},x^{31},x^{63}\}.
\]
\qed
\end{proof}

We remark that the polynomials $1+x+x^7$ and $1+x^3+x^7$ are primitive over $\GF(2^7)$.

\begin{lemma}\label{subgroup}
A set $\{0,1,x,x^3,x^7,x^{15},x^{31},x^{63}\}$ of distinct elements of $\GF(2^7)$ is a subgroup of the additive group of $\GF(2^7)$ if and only if either $1+x=x^7$, or $1+x^3=x^7$.  
\end{lemma}
\begin{proof}
Setting $1+x=x^7$ we can easily construct the following Cayley table:
\[
\begin{array}{l||l|l|l|l|l|l|l|l}
+ & 0 & 1 & x & x^3 & x^7 & x^{15} & x^{31} & x^{63} \\ \hline \hline
0 & 0 & 1 & x & x^3 & x^7 & x^{15} & x^{31} & x^{63} \\ \hline
1 & 1 & 0 & x^7 & x^{63} & x & x^{31} & x^{15} & x^3 \\ \hline
x & x & x^7 & 0 & x^{15} & 1 & x^3 & x^{63} & x^{31} \\ \hline
x^3 & x^3 & x^{63} & x^{15} & 0 & x^{31} & x & x^7 & 1 \\ \hline
x^7 & x^7 & x & 1 & x^{31} & 0 & x^{63} & x^3 & x^{15} \\ \hline
x^{15} & x^{15} & x^{31} & x^3 & x & x^{63} & 0 & 1 & x^7 \\ \hline
x^{31} & x^{31} & x^{15} & x^{63} & x^7 & x^3 & 1 & 0 & x   \\ \hline
x^{63} & x^{63} & x^3 & x^{31} & 1 & x^{15} & x^7 & x & 0  \\ 
\end{array}
\]
It is clear that all necessary conditions are satisfied.  A similar Cayley table can be constructed in the case $1+x^3=x^7$, which is equivalent to the case $1+x=x^{31}$.  We conclude this proof by showing that the remaining three cases $1+x=x^3$, $1+x=x^{15}$ and $1+x=x^{63}$ do not determine a group under the addition in $\GF(2^7)$.  Suppose that $1+x=x^3$.  Then $x^3+x^7=x^3(1+x^4)=x^3(1+x)^4=x^{15}$ and $x^7+x^{15}=x^7(1+x^8)=x^7(1+x)^8=x^{31}$, which implies that $x^3=x^{31}$, a contradiction.  If $1+x=x^{15}$ then $x^7+x^{15}=x^7(1+x)^8=x^7x^{120}=x^{127}=1$, implying $x=x^7$, a contradiction.  Finally, if $1+x=x^{63}$, we have $x+x^3=x(1+x)^2=xx^{126}=1$ which implies $x^3=x^{63}$, again a contradiction.      
\qed
\end{proof}

In the following Lemma we consider degree $4$ maximal arcs of Denniston type, containing the conic $C_1: x^2+xy+y^2+z^2=0$, in the standard pencil.  We note that every degree $4$ maximal arc of Denniston type is isomorphic to one in the standard pencil containing $C_1$.

\begin{lemma}  \label{CarIsomClass}
The number of conics in the standard pencil of $\PG(2,2^7)$ generating, together with  $C_1$, a degree $4$ Denniston arc of a given isomorphism type  is exactly $42$.
\end{lemma}
\begin{proof}
Consider the standard pencil.  Consider the conic $C_1: x^2+xy+y^2+z^2=0$.  This is the conic containing the point $(0,1,1)$ on the line $x=0$. Furthermore the point $(0,0,1)$ is the nucleus of $C_1$, and $(0,1,0)$ is the intersection point of the lines $x=0$ and $z=0$.  This means that so far $3$ points on the line $x=0$ are taken.  The other $126$ points on that line are contained in the $126$ conics left in the standard pencil.  Of course, every one of those conics together with $C_1$ gives rise to a unique degree $4$ maximal arc of Denniston type.  Since the third conic in such a $4$-arc is determined, we find that there are $63$ degree $4$ arcs of Denniston type in the pencil.  We also know that there are exactly $3$ isomorphism classes of degree $4$-arcs of Denniston type in $\PG(2,2^7)$, each of which has an automorphism group isomorphic to $C_{q+1}\rtimes C_2$ (see Lemma 4, Remark 1 and Lemma 5 of \cite{geometricapproach}).  It follows that there are $21$ degree $4$-arcs in each class, or equivalently, that there are $42$ conics in the standard pencil generating together with $C_1$ a degree $4$ Denniston arc of a given isomorphism type.    \qed    
\end{proof}

We already proved that the set of elements 
\[
\{0,1,\sqrt{w}^{-\sigma}, \sqrt{w}^{-(\sigma^2+\sigma)}, \sqrt{w}^{-(\sigma^3+\sigma^2+\sigma)},...,\sqrt{w}^{-(\sigma^6+...+\sigma^2+\sigma)}\}
\]
can be written as $A=\{0,1,x,x^3,x^7,x^{15},x^{31},x^{63}\}$, with $x$ a function of $w$.   Moreover we know that $A$ forms a group under the addition in $\GF(2^7)$ if and only if either $1+x=x^7$ or $1+x^3=x^7$.  It is clear that the set of solutions of the equation $1+x=x^7$ and the set of solutions of the equation $1+x^3=x^7$ have to be disjoint, otherwise $x=x^3$, a contradiction.  
This implies that in each of the two cases we have $7$ possible values for $x$. For every given non-trivial field automorphism $\sigma$, each of these values of $x$ yields a unique value of $w$.  In other words, we have $2\times7\times6=84$ possible values of $w$, that is we get $84$ conics which, together with the automorphism $\theta_{t,\sigma}$, possibly give rise to a Singer $8$-arc.  Note that, since $1+x+x^7$ and $1+x^3+x^7$ are not conjugate under any field automorphism the $84$ values of $w$ are indeed distinct. Furthermore note that, in view of Lemma \ref{sigmatox} and Lemma \ref{subgroup}, the $\sigma$ in $\theta_{t,\sigma}$ is uniquely determined once we have chosen a specific value of $w$ out of these $84$.\\
In view of Lemma \ref{CarIsomClass} we also can conclude that these $84$ conics together with $C_1$ determine exactly two isomorphism classes of degree $4$ maximal arcs of Denniston type in the standard pencil. In other words, at most two of the three isomorphism types of degree $4$ maximal arcs of Denniston type in $\mathrm{PG}(2,2^7)$ can possibly be extended to a Singer $8$-arc.  \\

\section{Necessary and sufficient condition}

We start by proving a lemma that provides us with a necessary and sufficient condition in order for $\theta_{t,\sigma}$ to generate a Singer $8$-arc.

\begin{lemma}  \label{NecAndSuf}
Let $\D=\{C_1,C_2,C_3\}$ be a $4$-arc of Denniston type in $\mathrm{PG}(2,2^7)$. Let 
Let $\theta$ be an automorphism of $\mathrm{PG}(2,2^7)$ with the properties that $C_2^\theta=C_1$,  $C_4:=C_1^\theta$ is disjoint from $C_1, C_2$ and $C_3$, and that $C_4$ has the same nucleus as $C_1,C_2$ and $C_3$. If $\D^{\theta^2}$ intersects both $\D$ and $\D^{\theta}$ in a conic, then $\D$ together with $\theta$ generate a Singer 8-arc, and consequently $\theta$ has order divisible by $7$.
\end{lemma}
\begin{proof}
Set $\D=\{C_1,C_2,C_3\}$, $\D^\theta=\{C_1,C_4,C_5\}$, with $C_2^\theta=C_1$, $C_1^\theta=C_4$ and $C_3^\theta=C_5$. So clearly $C_1\oplus C_2=C_3$ and $C_1\oplus C_4=C_5$. Let $\langle \D,C_4\rangle$ denote the 8-arc generated by $\D$ and $C_4$.
\medskip

As $C_1\in D^\theta$, we see that $C_4=C_1^\theta \in \D^{\theta^2}$. There are two possible cases (recall that $\D^{\theta^2}$ intersects $\D$ in a conic which has to be distinct from $C_1$):
\begin{enumerate}
\item $\D^{\theta^2}=\{C_4,C_2,C_2\oplus C_4=:C_6\}$,
\item $\D^{\theta^2}=\{C_4,C_3,C_3\oplus C_4=:C_7\}$.
\end{enumerate}
We discuss both cases separately. Note that all $\oplus$-additions and related computations are well defined by Lemma \ref{aggiko} and Theorem \ref{synthmathon}.

\begin{enumerate}
\item {\bf The case $\D^{\theta^2}=\{C_4,C_2,C_2\oplus C_4=:C_6\}$.}

From $C_4\in \D^\theta$ it follows that $C_4^\theta \in \D^{\theta^2}$. Clearly $C_4^\theta \neq C_4$. We will show that  $C_4^\theta$ can also not be $C_2$. This would clearly yield an automorphism of order a power of $3$ stabilizing the the $8$-arc 
$\langle \D,C_4 \rangle$. Hence $\theta$ necessarily would stabilize one of the conics in this $8$-arc. But then there has to be a line that is not fixed pointwise containing at least $3$ fix points, and so $\theta\in\mathrm{P\Gamma L}(3,2^7)\setminus\mathrm{PGL}(3,2^7)$. Consequently $7$ divides the order of $\theta$, a contradiction. Hence $C_4^\theta=C_6$. As $C_2^\theta=C_1$ we obtain $\D^{\theta^3}=\{C_1,C_6, C_1\oplus C_6=:C_7\}$.  We need to show that $\D^{\theta^3}\in\langle \D,C_4\rangle$. But this is true since $C_7=C_1\oplus C_6= C_1\oplus C_2 \oplus C_4=C_3\oplus C_4 \in \langle \D,C_4 \rangle$. 


Next we look at $\D^{\theta^4}$. Note that from the previous step it follows that $C_6^\theta=C_7$, and hence that $\D^{\theta^4}=\{C_4,C_7,C_4\oplus C_7\}$. But $C_4\oplus C_7=C_4\oplus C_3\oplus C_4=C_3$. And so $\D^{\theta^4}=\{C_4,C_7,C_3\}\in \langle\D,C_4\rangle$.


Consequently $\D^{\theta^5}=\{C_6,C_3,C_5\}\in \langle \D,C_4\rangle$, $\D^{\theta^6}=\{C_7,C_5,C_2\}\in \langle \D,C_4\rangle $, and $\D^{\theta^7}=\{C_3,C_2,C_1\}=\D$.

It is now also clear that the action of $\theta$ on the conics of $\langle \D,C_4\rangle $ is described by $C_1\rightarrow C_4 \rightarrow C_6 \rightarrow C_7 \rightarrow C_3 \rightarrow C_5 \rightarrow C_2 \rightarrow C_1$. Hence $\D$ and $\theta$ generate a unique Singer 8-arc, and $\theta$ has order divisible by $7$.

\item {\bf The case $\D^{\theta^2}=\{C_4,C_3,C_3\oplus C_4=:C_7\}$.}

First assume that $C_4^\theta=C_7$. But then $C_5^\theta=C_3$, from which $C_3^{\theta^2}=C_3$, yielding a contradiction as in the case $C_4^\theta=C_2$ above. Hence this case cannot occur, and consequently $C_4^\theta=C_3$.

We quickly see that $\D^{\theta^3}=\{C_3,C_5,C_3\oplus C_5=:C_6\}$. Now $C_6=C_3\oplus C_5=C_3\oplus C_1\oplus C_4=C_2\oplus C_4$, and hence $\D^{\theta^3}\in \langle \D,C_4 \rangle$.

From $C_3^\theta=C_5$ and $C_5^\theta=C_7$ it follows that $\D^{\theta^4}=\{C_5,C_7,C_6^\theta=C_5\oplus C_7\}$. But $C_6^\theta=C_5\oplus C_7=C_1\oplus C_4\oplus C_3\oplus C_4=C_2$, and so $\D^{\theta^4}\in \langle \D,C_4\rangle $.

Consequently $\D^{\theta^5}=\{C_7,C_6,C_1\}\in \langle \D,C_4\rangle $, $\D^{\theta^6}=\{C_6,C_2,C_4\}\in \langle \D,C_4\rangle $, and $\D^{\theta^7}=\{C_2,C_1,C_3\}=\D$.

It is now also clear that the action of $\theta$ on the conics of $\langle \D,C_4\rangle$ is described by $C_1 \rightarrow C_4 \rightarrow C_3 \rightarrow C_5 \rightarrow C_7 \rightarrow C_6 \rightarrow C_2 \rightarrow C_1$. Hence $\D$ and $\theta$ generate a unique Singer 8-arc, and $\theta$ has order divisible by $7$.
\end{enumerate}
\qed
\end{proof}

\begin{remark}
{\rm
As mentioned in Section \ref{NecessCond}, if a Mathon arc is to be a Singer $8$-arc, it can be constructed (or at least it is isomorphic to one that can be constructed) from a Denniston $4$-arc $\D$ in the standard pencil containing the conic $C_1: x^2+xy+y^2+z^2=0$ together with an appropriate automorphism $\theta_{t,\sigma}$. Such automorphism clearly would have to satisfy the conditions of Lemma \ref{NecAndSuf}, and so Lemma \ref{NecAndSuf} provides us with necessary and sufficient conditions on $\theta_{t,\sigma}$ in order to give rise to a Singer $8$-arc. This means that theoretically the necessary subgroup-condition analyzed in Lemma \ref{sigmatox} and Lemma \ref{subgroup} would also follow from the above necessary and sufficient condition. However, we believe that first dealing with the subgroup-condition as we did, makes the analysis of the above necessary and sufficient condition easier, and further provides insightful information on the Singer $8$-arcs that will arise.
}
\end{remark}

\begin{remark}  \label{t-values}
{\rm
Let $\theta_{t,\sigma}$ be an automorphism as considered in  (\ref{automw}). Suppose furthermore that $\theta_{t,\sigma}$ gives rise to a Singer $8$-arc. As mentioned in Lemma \ref{iota} 
there is a unique involution $\iota$ stabilizing all conics of the arc. Hence also $\theta'_{t,\sigma}:=\iota  \theta_{t,\sigma}=\theta_{t,\sigma} \iota$ will be an automorphism giving rise to the same Singer $8$-arc as $\theta_{t,\sigma}$.	This involution is easily seen to be induced by 
$$E=\left(\begin{array}{ccc} 
		1&0&0\\1&1&0\\ 0&0&1
	\end{array}\right).$$
And consequently $\theta'_{t,\sigma}:x\mapsto EAx^\sigma$, where
$$EA=\left(\begin{array}{ccc} 
		\sqrt{w}^{-\sigma}&0&0\\t+\sqrt{w}^{-\sigma}&\sqrt{w}^{-\sigma}&0\\ \sqrt{\sqrt{w}^{-\sigma}t+t^2}&0&1
	\end{array}\right).$$
Thus $\theta'_{t,\sigma}=\theta_{t+\sqrt{w}^{-\sigma},\sigma}$.	This implies that the $t$-values corresponding to a given Singer $8$-arc come in pairs, $t$ and $t+\sqrt{w}^{-\sigma}$. 
}
\end{remark}

\begin{remark}  \label{Denniston}
{\rm
In our analysis so far we have nowhere required that the hypothetical Singer $8$-arc would be a proper Mathon arc. Hence, some of the Singer $8$-arcs we will discover in what follows may well be arcs of Denniston type.  It is however, given $\theta_{t,\sigma}$, easy to decide whether an arc is of Denniston or proper Mathon type. To be of Denniston type all conics of the arc should be contained in the standard pencil, and hence all degree $4$ maximal arcs in the considered $8$-arc should have the same line at infinity, namely $z=0$. Consequently this line should be fixed by $\theta_{t,\sigma}$. This happens if and only if  the element on position $(3,1)$ of matrix $A$ is equal to zero, or equivalently 
$\sqrt{w}^{-\sigma}t+t^2=0$. Hence, if and only if $t=0$ or $t=\sqrt{w}^{-\sigma}$. In view of the previous remark, we see that both of these $t$-values will correspond to one and the same Denniston $8$-arc.
}
\end{remark}
 
 We are now ready to start exploiting Lemma \ref{NecAndSuf}.

Assume that the same settings as presented in Section \ref{NecessCond} hold, that is, the additive subgroup $\{0,1,w,w+1\},w \in \GF(2^7) \setminus \{0,1\}$ is the one associated to the maximal arc $\D$.  The conics with equation
\[
x^2+xy+y^2+iz^2=0,
\]
where $i=1,w,w+1$, that are contained in $\D$ are denoted by $C_1, C_2$ and $C_3$ respectively.  Next, let $\theta=\theta_{t,\sigma}$ be an automorphism of $\PG(2,2^7)$ as defined in (\ref{automw}).  

Instead of choosing $w$ to be one of the $84$ values found in Section \ref{NecessCond}, we will instead fix $\sigma=2$. In view of Lemma \ref{sigmatox} we can do so without loss of generality. Once $x$ is known, this will determine $w$ uniquely.

Hence $\theta:p\mapsto Ap^2$, with
$$A=\left(\begin{array}{ccc} 
		w^{-1}&0&0\\t+w^{-1}&w^{-1}&0\\ \sqrt{w^{-1}t+t^2}&0&1
	\end{array}\right).$$

Suppose that $\D^\theta=\{C_1,C_4,C_5\}$, with $C_2^\theta=C_1$, $C_1^\theta=C_4$ and $C_3^\theta=C_5$.  Due to the proof of the previous lemma we need to consider two specific cases which possibly can lead to a Singer $8$-arc.  Either $C_1^{\theta^2}=C_3$ or $C_1^{\theta^2}=C_6$, where $C_6:=C_2 \oplus C_4$.  We will investigate both cases separately.  \\

Let $p=(x,y,1)$ be a general point of the conic $C_1$.    We know that the automorphism $\theta^2$ is determined by the matrix $A.A^2$, and the automorphism $\sigma^2=4$.  Using this we are able to compute the point $p^{\theta^2}$.  This gives us
\begin{equation}  \label{p^theta^2} 
p^{\theta^2}=\left(\begin{array}{c} 
		w^{-3}x^4\\(w^{-2}t+w^{-1}t^2)x^4+w^{-3}x^4\\(w^{-2}\sqrt{w^{-1}t+t^2}+(w^{-1}t+t^2))x^4+1
	\end{array}\right).
\end{equation} 

We start with the case $C_1^{\theta^2}=C_3$.  This means that we want $p^{\theta^2}$ to be contained in $\C_3$.  Expressing this condition yields the following equation:
\begin{eqnarray*}
&&w^{-6}x^8+(w^{-3}x^4)((w^{-2}t+w^{-1}t^2)x^4+w^{-3}x^4)+(w^{-4}t^2+w^{-2}t^4)x^8+w^{-6}y^8  \\
&&+(w+1)((w^{-4}(w^{-1}t+t^2)+w^{-2}t^2+t^4)x^8+1)=0.
\end{eqnarray*}

Using the fact that $x^2+xy+y^2+1=0$ we can simplify the previous equation to
\begin{equation}  \label{equationC3}
((w^{-2}+w+1)t^4+(w^{-4}+w^{-3}+w^{-2}+w^{-1})t^2+w^{-4}t)x^8+w^{-6}+w+1=0.
\end{equation}
As this should hold for any point $p$ on $C_1$, this equation should be identically zero.  This means that both the coefficients of $x^8$ and $x^0$, which are $(w^{-2}+w+1)t^4+(w^{-4}+w^{-3}+w^{-2}+w^{-1})t^2+w^{-4}t$ and $w^{-6}+w+1$ respectively, have to be 0.  First, we have a look at the condition
\[
w^{-6}+w+1=0.
\]
With the notation used in Lemma \ref{sigmatox} and the fact that $\sigma=2$ we find that $l=1$ which implies that $x=w^{-1}$.  
We easily find  that $w^{-6}+w+1=0$ if and only if $x^7+x+1=0$.  In other words, the case where $x^7+x+1=0$ is the only possible case that allows the coefficient of $x^0$ in (\ref{equationC3}) to be 0.  Furthermore 
\[
(w^{-2}+w+1)t^4+(w^{-4}+w^{-3}+w^{-2}+w^{-1})t^2+w^{-4}t=0
\]
has to hold.  This will provide us with four values for $t$ which are $t=0$, $t=w^{-1}$, $t=w^{115}$ and $t=w^{39}$.  From Remark \ref{t-values} and Remark \ref{Denniston} it is clear that the two solutions $t=0$ and $t=w^{-1}$ will lead to a degree $8$ maximal arc of Denniston type.  The two other values for $t$ will extend the degree $4$ maximal arc $\D$ to a unique Singer $8$-arc.  \\

Next, we move on to the second case: $C_1^{\theta^2}=C_6$.  We now aim for $p^{\theta^2}$ to be contained in $\C_6$.  First of all we have to compute $C_4$ which is the image of $C_1$ under $\theta$.  After some calculations we find
\[
C_4:(1+(w+w^{-1})t+(w^2+1)t^2)x^2+xy+y^2+w^{-2}z^2=0.
\]
Since we know that $C_2:x^2+xy+y^2+wz^2=0$ we can now determine the equation of the conic $C_6:=C_2 \oplus C_4$.  We get
\[
C_6:(w+(1+(w+w^{-1})t+(w^2+1)t^2)w^{-2})x^2+(w+w^{-2})xy+(w+w^{-2})y^2+(w^2+w^{-4})z^2=0.
\]

Using (\ref{p^theta^2}) the condition $p^{\theta^2} \in C_6$ can be expressed in the following way:
\begin{eqnarray*}
&&(w+(1+(w+w^{-1})t+(w^2+1)t^2)w^{-2})(w^{-3}x^4)^2  \\  
&&+(w+w^{-2})(w^{-3}x^4)(w^{-2}tx^4+w^{-1}t^2x^4+w^{-3}y^4)  \\
&&+(w+w^{-2})((w^{-4}t^2+w^{-2}t^4)x^8+w^{-6}y^8)  \\
&&+(w^2+w^{-4})(w^{-4}(w^{-1}t+t^2)+w^{-2}t^2+t^4)x^8+1)=0.
\end{eqnarray*}

After some calculation and again using the fact that $x^2+xy+y^2+1=0$ this equation can be simplified to
\begin{equation}  \label{equationC6}
((w^2+w^{-1})t^4+(w^{-2}+1)t^2+(w^{-4}+w^{-3})t)x^8+w^{-8}+w^{-5}+w^{-4}+w^2=0.
\end{equation}

Analogous to the first case this equation should be identically zero.  We start by checking if the coefficient of $x^0$ can be equal to 0 and, since $x=w^{-1}$ with the notation of Lemma \ref{sigmatox}, we see that

\[  w^{-8}+w^{-5}+w^{-4}+w^2=0  
\Leftrightarrow x^{10}+x^7+x^6+1=0.
\]  
 
Now assume that $x^7+x^3+1=0$.  In this case we find that $x^{10}+x^7+x^6+1=x^{10}+x^6+x^3=x^3(x^7+x^3+1)=0$, exactly what we wanted.  On the other hand, suppose that $x^7+x+1=0$ holds.  This would imply that 
\begin{eqnarray*}
x^{10}+x^7+x^6+1&=&x^7(x^3+1)+(x^3+1)^2  \\
&=&(x^3+1)(x^7+x^3+1)  \\
&=&(x^3+1)(x^3+x)  \\
&=&x(x^3+1)(x^2+1).
\end{eqnarray*}

But since $x \neq 0$, $x^3 \neq 1$ and $x^2 \neq 1$ this can never be 0.  In other words the case in which $x^7+x^3+1=0$ is the only possible case that allows the coefficient of $x^0$ to be 0.  Finally we have a look at the identity
\[
(w^2+w^{-1})t^4+(w^{-2}+1)t^2+(w^{-4}+w^{-3})t=0,
\]
assuring that also the coefficient of $x^8$ in (\ref{equationC6}) is 0.  The four solutions satisfying this equation are $t=0$, $t=w^{-1}$, $t=w^{91}$ and $t=w^{8}$.  Analogous to what we have seen above the two solutions $t=0$ and $t=w^{-1}$ yield a degree $8$ maximal arc of Denniston type.  The other two values for $t$ that satisfy this equation will lead to a unique Singer $8$-arc.  \\

We can conclude the previous findings by saying that in both cases $x^7+x+1=0$ and $x^7+x^3+1=0$ the degree $4$ maximal arc $\D$ can uniquely be extended to a Singer $8$-arc.  
\medskip

We end this section by providing actual equations of the two Singer $8$-arcs in $\PG(2,2^7)$.  Let $a$ be a primitive element  of $\GF(2^7)$.  Both the Singer $8$-arcs clearly can be given by the set
\begin{equation}\label{SingerEq}
\{(x,y,z) \in \PG(2,2^7) \vert a^ix^2+xy+y^2+a^jz^2=0\} \cup \{(0,0,1)\},
\end{equation}
where there are seven ordered pairs $(i,j)$, each corresponding with one of the conics of the arc.
There are two cases.
The unique, up to isomorphism, Singer $8$-arc in the case where $1+a+a^7=0$ is the one with exponents 
\begin{equation}\label{Singer1}
(i,j)=(0,0),(0,-1),(0,6),(16,2),(39,14),(93,62),(101,30). 
\end{equation}
In the other case where $a$ satisfies $1+a^3+a^7=0$ the unique, up to isomorphism, Singer $8$-arc is the one with exponents 
\begin{equation}\label{Singer2}
(i,j)=(0,0),(0,-1),(0,30),(18,2),(12,62),(33,6),(43,14).
\end{equation}
These values for $(i,j)$ can be obtained by actually computing the morphism $\theta_{t,\sigma}$ using the above, and then letting act this morphism on $C_1$. This can  easily be done using a computer algebra package such as GAP.

\section{The count}  \label{8ArcCount}

In this section we will count the number of Singer $8$-arcs and the number of normal $8$-arcs in $\PG(2,2^7)$.  Since these are the only two classes of maximal $8$-arcs of Mathon type in $\PG(2,2^7)$ it will lead to the total number of degree $8$ maximal arcs of Mathon type in $\PG(2,2^7)$. This will fill the hole that was left in \cite{geometricapproach}.
In the following lemma the number of Singer $8$-arcs of proper Mathon type in $\PG(2,2^7)$ is counted.

\begin{lemma}  \label{SingerCount}
There are, up to isomorphism, exactly $2$ Singer $8$-arcs in $\PG(2,2^7)$.
\end{lemma}
\begin{proof}
At the end of section \ref{NecessCond} we concluded that at most two out of the three isomorphism types of degree $4$ maximal arcs of Denniston type in $\PG(2,2^7)$ could possibly be extended to a Singer $8$-arc.  In the previous section it became clear that both of these isomorphism classes induce a unique Singer $8$-arc.    \qed
\end{proof}

\begin{lemma}  \label{NormalCount}
The number of non-isomorphic normal $8$-arcs in $\PG(2,2^7)$ is $199$.
\end{lemma}
\begin{proof}
This proof is quite analogous to the proof of Theorem $4$ in \cite{geometricapproach}.  Let $\D^1, \D^2$ and $\D^3$ be chosen fixed and representative of each of the three isomorphism classes of degree $4$ maximal arcs of Denniston type in the standard pencil.  Assume $\D^i$ consists of the conics $C_1, C_2^i$ and $C_3^i$, $i=1,2,3$.  We will start by counting how many degree $8$ maximal arcs of Mathon type contain one of the three degree $4$ maximal arcs, say $\D^1$, have the line $x=0$ as elation axis and $(0,1,0)$ as the intersection point of their lines at infinity. This is done as follows. Every conic disjoint from $\D^1$ and having the same nucleus as $\D^1$ will extend $\D^1$ in a unique way to a maximal arc of Mathon type (which may be Denniston). On the other hand, every Mathon arc of degree $8$ that contains $\D^1$ will give rise to four such conics.  Now each such conic, together with $C_1$, generates a unique maximal arc of dgree $4$ (of Denniston type) which has to be isomorphic to one of $\D^i$, $i=1,2,3$. Hence it will be sufficient to count in how many ways we can map $\D^i$, $i=1,2,3$, on an isomorphic degree $4$ arc which intersects $\D^1$ exactly in $C_1$ plus the nucleus, and which line at infinity contains $(0,1,0)$. \\

\begin{itemize}
\item Assume $i \neq 1$.  \\
Clearly we need an automorphism $\theta$ such that $(\D^i)^\theta$ satisfies the properties described above. Hence $\theta$ has to map one of $C_1$ or $C_j^i$, $j=2,3$ onto $C_1$. It follows that $\theta$ is of the form (\ref{automw}), with $w$ the value corresponding to $C_1$ or $C_j^i$ respectively. 
The above conditions will be satisfied if and only if  the trace conditions (\ref{Tr1}) and (\ref{Tr2}) seen above, are satisfied.  These two conditions can be written as 

\begin{eqnarray*}
\left\{
	\begin{array}{l} 
		\Tr[A_1(w,\sigma)t]=0\\ \Tr[A_2(w,\sigma)t]=0,
	\end{array}
\right.
\end{eqnarray*}
where $A_1$ and $A_2$ are both functions of $w$ and $\sigma$.  We actually obtain two linear equations that correspond to two hyperplanes in the vector space $V(7,2)$.  In \cite{geometricapproach} it was shown that these hyperplanes are distinct. This means that for every $w$ and every field automorphism $\sigma$ there are $2^5$ solutions for $t$.  As noticed in  \cite{geometricapproach}, and as seen in the previous section, these $t$-values always come in pairs.  This implies that, for every $w$ and $\sigma$, there are $2^4$ degree $4$ maximal arcs.  One of them gives rise to a degree $8$ maximal arc of Denniston type, and so there are
\[
3\cdot7\cdot(2^4-1)
\] 
automorphisms $\theta$ that satisfy the needed conditions and induce a degree $8$ maximal arc of proper Mathon type.  Of course, one such automorphism leads to two conics disjoint from $\D^1$ and so we get
\[
2\cdot3\cdot7\cdot(2^4-1)=630
\] 
conics that extend $\D^1$ to a degree $8$ maximal arc of Mathon type.  \\
\item Now assume $i=1$.  \\
In the cases where $C_2^1$ is mapped onto $C_1$ and $C_3^1$ is mapped onto $C_1$ we also find
\[
2\cdot7.(2^4-1)
\]
conics to extend $\D^1$.  If we now consider the case where $C_1$ is fixed however, we have to make sure that $\sigma$ is not the identity since this would result in conics which are not disjoint (see Remark 2. in \cite{geometricapproach}).  And so when $i=1$ we find
\[
2\cdot2\cdot7\cdot(2^4-1)+2\cdot6\cdot(2^4-1)=600
\]
conics to extend $\D^1$ in this case. 
\end{itemize}

 Finally, as there were two choices for $\D^i$, $i \neq 1$, there are a total of 
\[
2\cdot630+600=1860
\]
 conics that extend $\D^1$ to a proper Mathon arc satisfying the desired properties.  Of course, there were three choices for $\D^1$ and so we get
\[
3\cdot1860=5580
\]
conics that will extend some $\D^i$ to a degree $8$ maximal arc of Mathon type.  However, due to Lemma \ref{SingerCount} we know that two out of the three isomorphism classes of degree $4$ maximal arcs can be extended to a unique Singer $8$-arc.  This means that $8$ conics will extend some $\D^i$ to a Singer $8$-arc which implies that there are actually
\[
5580-8=5572
\]
conics that will extend some $\D^i$ to a normal $8$-arc.  Since the four conics disjoint from $\D^i$ in such a normal $8$-arc all give rise to the same $8$-arc there are $1393$ normal $8$-arcs that contain some degree $4$ maximal arc $\D^i$.  Hence, the number of non-isomorphic normal $8$-arcs is
\[
1393/7=199.
\]
The fact that we divide by $7$ is a consequence of Corollary 1 and Lemma 7 in \cite{geometricapproach}, which states that there is a unique isomorphism of the plane mapping a degree $4$ Denniston arc in a normal Mathon $8$-arc onto one of $\D^i$.  
\qed \end{proof}

The total number of non-isomorphic degree $8$ maximal arcs of Mathon type is now easily calculated and yields the following theorem.

\begin{theorem}  \label{Total8Count}
The number of non-isomorphic degree $8$ maximal arcs of proper Mathon type in $\PG(2,2^7)$ is equal to $201$, two of which are Singer $8$-arcs, and $199$ of which are normal $8$-arcs.
\end{theorem}

\section{Bigger fields}  \label{equations}

It turns out that  Singer $8$-arcs also exist in bigger fields.  

Consider $\GF(2^h)$, with $h$ odd and $7 \> \vert \> h$.  Let $\mathrm{TR}$ denote the absolute trace map from the finite field $\GF(2^h)$ onto $\GF(2)$ and let $\mathrm{tr}$ be the  absolute trace map from the field $\GF(2^7)$ onto $\GF(2)$.  Now, since $h$ is odd and a multiple of $7$ the equality $\mathrm{TR}(\alpha)=\mathrm{tr}(\alpha)$ will hold for every $\alpha \in \GF(2^7)$, subfield of $\GF(2^h)$.  This implies that all conics from (\ref{SingerEq}), as well with exponents (\ref{Singer1}) as with exponents (\ref{Singer2}), are exterior to the line $z=0$. However, one has to be careful with the definition of Singer arc over these bigger fields. To see this,  consider first the case where $l\neq7k$ for some (odd) $k$. In this case consider the smallest positive $t$ such that $lt\equiv1\pmod{7}$. Then the field automorphism $\tau=2^{lt}$ acts as squaring on the subfield $\GF(2^7)$ and has order $7$ (over the big field). If we now replace the automorphism $\sigma=2$ from Section 3 by $\tau$, then we can easily see that we produce two Mathon $8$-arcs that admit a cyclic group of order $7$ acting sharply transitively on the $7$ conics of the arc. These are clearly Singer $8$-arcs in the obvious sense.

If however we consider the case where $l=7k$ for some (odd) $k$, then there is no field automorphism of $\GF(2^l)$ that acts as squaring on the subfield $\GF(2^7)$ and has order $7$ (over the big field). In this case every field automorphism that acts as squaring on the subfield $\GF(2^7)$ will have as order a proper power of $7$. In this case the arcs we produce will only admit a cyclic group acting transitively on the $7$ conics of the arc, but not sharply transitively, as the elements of the unique cyclic subgroup of order seven will stabilize all conics of the arc (but not fix them pointwise). One could define such arcs as {\it Singer $8$-arcs of the second kind}.

\section{Open questions}

The following two questions seem now to be natural.

\begin{itemize}

\item Can we construct Singer arcs of degree bigger than $8$, that is, are there for example degree $16$ arcs that admit a (cyclic) automorphism group acting sharply transitively on their conics? If so, over which fields do these arcs exist? What about Singer arcs of the second kind?

\item Proper $8$-arcs of Mathon type define (considered conicwise)  the Fano-plane. Furthermore, the Singer group of the Mathon Singer $8$-arcs acts as a Singer group on this Fano-plane. This Singer group is only a subgroup of the full automorphism group of the Fano-plane. Are there fields over which there exist Mathon $8$-arcs that admit the full automorphism group of the Fano-plane (in its natural action on the conics of the arc)? If not, what is the largest subgroup of $\mathrm{GL}_3(2)$ that can occur, and over which fields does this happen? In such case one would of course not require that if a conic is stabilized by an automorphism it is fixed pointwise.

\end{itemize}

\begin{acknowledgements} 

The authors would like to thank the Mathematics department of UC San Diego, where part of this paper was written, and in particular Jacques Verstraete, for the warm hospitality. 

\end{acknowledgements} 
\bibliographystyle{plain}

\end{document}